\documentclass{article}
\usepackage{amssymb,latexsym,amsmath,epsfig,amsthm,hyperref,tipa,comment}
\newtheorem{ErdPom}{Theorem}
\newtheorem{generalinequality}[ErdPom]{Lemma}
\newtheorem{ErdPombounds}[ErdPom]{Lemma}
\begin{document}

\vspace*{-2cm}

\Large
 \begin{center}
On the length of an interval that contains distinct multiples of the first $n$ positive integers \\ 

\hspace{10pt}

\large
Wouter van Doorn \\

\hspace{10pt}

\end{center}

\hspace{10pt}

\normalsize

\vspace{-25pt}

\centerline{\bf Abstract}
Confirming a conjecture by Erd\H os and Pomerance, we prove that there exist intervals of length $\frac{cn\log n}{\log \log n}$ that do not contain distinct multiples of $1, 2, \ldots, n$. 

\section{Introduction}
Define $f(n, m)$ to be the least integer so that the interval $(m, m + f(n, m)]$ contains $n$ distinct integers $a_1, a_2, \ldots, a_n$ such that $i$ divides $a_i$ for all $i$. Erd\H os and Pomerance conjectured in \mbox{\cite{3}} that $\max_m f(n, m) - f(n, n)$ goes to infinity with $n$. In \mbox{\cite{2}} Erd\H os even offered $1000$ rupees for a solution, and it is now listed as (part of) problem \#711 at Bloom's website \cite{1}. In this short note we will settle their conjecture in the affirmative by proving the following theorem.

\begin{ErdPom} \label{main}
We have the lower bound $$\max_m f(n, m) - f(n, n) > \frac{0.36n \log n}{\log \log n}$$ for all large enough $n \in \mathbb{N}$. In particular, if $n$ is sufficiently large, then an interval of length $\frac{0.36n \log n}{\log \log n}$ exists that does not contain distinct multiples of $1, 2, \ldots, n$. 
\end{ErdPom}

We note that the second sentence of Theorem \ref{main} immediately follows from the first, as we trivially have $f(n, n) \ge 0$. 

\section{Proof of Theorem \ref{main}}
The proof of Theorem \ref{main} is based on the following fairly straight-forward, but surprisingly powerful, lemma.

\begin{generalinequality} \label{genineq}
For all positive integers $k$ and $n$ we have 

\begin{equation} \label{geninq}
kn + f(kn, kn) \le k^2n + f(n, k^2n). 
\end{equation}
\end{generalinequality}

\begin{proof}
Replacing both $n$ and $m$ in the definition of $f(n, m)$ by $kn$, we need to show that for every $1 \le i \le kn$ there is a multiple $a_i$ of $i$ with $a_i \in (kn, k^2n + f(n, k^2n)]$, where all $a_i$ are distinct. Now, for every $i \in (n, kn]$ we simply choose $a_i = ki \in (kn, k^2n]$, which is certainly divisible by $i$, while all $a_i$ are distinct as $k$ is non-zero. On the other hand, by the definition of $f(n, m)$ with $m = k^2n$, it is for all $i \in [1, n]$ possible to choose distinct multiples $a_i \in (k^2n, k^2n + f(n, k^2n)]$. By combining the disjoint intervals we conclude that all $a_i$ are indeed contained in $(kn, k^2n + f(n, k^2n)]$.
\end{proof}

To apply Lemma \ref{genineq}, we will need lower and upper bounds on $f(n, n)$.

\begin{ErdPombounds} \label{erdpombounds}
For all sufficiently large $n \in \mathbb{N}$ we have the following inequalities:

$$\left(\frac{2}{\sqrt{e}}+o(1)\right)n \sqrt{\frac{\log n}{\log \log n}} < f(n, n) < \big(2 + o(1)\big)n \sqrt{\log n}.$$
\end{ErdPombounds}

Both the lower and the upper bound were already proven by Erd\H os and Pomerance in \mbox{\cite{3}}. With these bounds we are ready to prove Theorem \ref{main}.

\begin{proof}[Proof of Theorem \ref{main}]
With $n$ a sufficiently large integer, define $k := \left \lceil 0.6\sqrt{\frac{\log n}{\log \log n}} \hspace{3pt} \right \rceil$ and choose $\epsilon := \frac{1}{100}$. Using the inequality $\frac{2}{\sqrt{e}} > 1.21$ and the fact that $n$ is sufficiently large, the bounds from Lemma \ref{erdpombounds} then imply, in particular, that

\begin{equation} \label{eqlower}
f(kn, kn) > (2 + \epsilon)k^2n 
\end{equation}

\noindent and 

\begin{equation} \label{equpper}
\epsilon k^2n > f(n, n).
\end{equation}

Combining Equations (\ref{geninq}), (\ref{eqlower}), and (\ref{equpper}) now finishes the proof. Indeed,
\begin{align*}
\max_m f(n, m) &\ge f(n, k^2n) \\
&\ge kn + f(kn, kn) - k^2n \\
&> (2+\epsilon)k^2n - k^2n \\
&= \epsilon k^2n + k^2n \\
&> f(n, n) + \frac{0.36n \log n}{\log \log n}. \qedhere
\end{align*}
\end{proof}

\vskip 20pt\noindent {\bf Acknowledgements.}
The author would like to express his gratitude to Terence Tao for various valuable comments on an earlier draft of this paper, and to Dan Wood for his generosity.


\begin{thebibliography}{1}\footnotesize

\bibitem{1}
	T. F. Bloom, Erd\H{o}s Problem \#711, \href{https://www.erdosproblems.com}{https://www.erdosproblems.com}.
	
\bibitem{2}
	P. Erd\H os,
	Some of my forgotten problems in number theory,
	{\it Hardy-Ramanujan J.} {\bf 15} (1992), 34--50.
	
\bibitem{3}
	P. Erd\H os and C. Pomerance,
	Matching the natural numbers up to $n$ with distinct multiples of another interval,
	{\it Indag. Math.} {\bf 83} (2) (1980) 147--161.
	
\end{thebibliography}
\end{document}